\documentclass[final,10pt]{amsart}
\usepackage{amsmath,amsfonts,amssymb,amscd,verbatim}
\usepackage[margin=0.4in]{geometry}
\usepackage{fullpage}
\usepackage{enumerate}

\usepackage{bm}
\usepackage{xcolor}
\usepackage{enumitem}

\numberwithin{equation}{section}
\theoremstyle{plain}
\newtheorem{thm}{Theorem} 
\newtheorem{cor}[thm]{Corollary}

\newtheorem{theorem}[equation]{Theorem} 

\newtheorem{corollary}[equation]{Corollary} 
\newtheorem{lemma}[equation]{Lemma}

\theoremstyle{definition}
\newtheorem{definition}[equation]{Definition} 
 
\newtheorem{example}[equation]{Example}

\DeclareMathOperator\ann{l{.}ann}
\DeclareMathOperator\Aut{Aut}
\DeclareMathOperator\chr{char}

\DeclareMathOperator\diag{diag}

\DeclareMathOperator\Hom{Hom}

\DeclareMathOperator\ord{ord}

\newcommand\NN{\mathbb N}

\newcommand\ZZ{\mathbb Z}

\newcommand\cR{\mathcal R}

\newcommand\bp{\mathbf p}
\newcommand\bx{\mathbf x}

\newcommand\gu{\underline g}
\newcommand\cu{\underline \chi}

\newcommand\kk{\Bbbk}

\newcommand\ep{\varepsilon}

\renewcommand{\int}{\mathrm{int}}
\newcommand\inv{^{-1}}

\newcommand\iso{\cong}

\newcommand\tensor{\otimes}
\newcommand\tornado{\xi}

\newcommand\grp[1]{\langle #1 \rangle}

\title{Prime and semiprime quantum linear space smash products}
\author[Gaddis]{Jason Gaddis}
\address{Miami University, Department of Mathematics, 301 S. Patterson Ave., Oxford, Ohio 45056} 
\email{gaddisj@miamioh.edu}


\begin{document}

\begin{abstract}
Bosonizations of quantum linear spaces are a large class of pointed Hopf algebras that include the Taft algebras and their generalizations. We give conditions for the smash product of an associative algebra with a bosonization of a quantum linear space to be (semi)prime. These are then used to determine (semi)primeness of certain smash products with quantum affine spaces. This extends Bergen's work on Taft algebras.
\end{abstract}

\maketitle

\section{Introduction}

Throughout, $\kk$ is an algebraically closed, characteristic zero field.
All algebras are associative, unital $\kk$-algebras. All tensor products are assumed to be over $\kk$.

An important property to study for any family of algebras
is how the prime and semiprime conditions are behaved with respect to various extensions. 
For example, Ore extensions of prime rings are prime \cite[Theorem 2.9]{MR}
and group rings $\kk G$ are prime if and only if $G$ has no nonidentity finite normal subgroup \cite[Theorem 1]{Pgrp}.
More generally, many authors have considered the question of (semi)primeness for smash products $A\#H$, where $H$ is a Hopf algebra and $A$ an $H$-module algebra \cite{B,CM,CF,FM,LM,lomp2,LPass,SVO}. The interested reader is directed to the survey article by Lomp for a thorough account of the state of the art with regards to semiprime smash products \cite{lomp}.
This paper concerns the question of (semi)prime extensions for certain pointed Hopf algebras $H$. In particular, we build on the work of Bergen who studied this question when $H$ is the Taft algebra \cite{B}. 

The Taft algebras were originally defined by Earl Taft \cite{T} in the study of finite-dimensional Hopf algebras whose antipode had arbitrarily large order. They are among the simplest examples within the classification of finite-dimensional pointed Hopf algebras by Andruskiewitsch and Schneider \cite{ASqls,AS1}. Pointed Hopf algebras have also been studied recently in the area of \emph{quantum symmetry}, or actions of Hopf algebras on noncommutative algebras \cite{cline,CG,EW1,EW2,GWY,KW}.




We begin by reviewing basic definitions and notions for our study, as well as a statement of our main results.

For a group $G$, we denote the \emph{character group} of $G$ by $\widehat{G}=\Hom(G,\kk^\times)$.
Throughout, all groups will be finite abelian and hence $\chi(g)$ is a root of unity for all $\chi \in \widehat{G}$ and all $g \in G$. We denote by $\ord(\chi(g))$ the order of $\chi(g)$ as an element in group $\kk^\times$.

Let $H$ be a Hopf algebra and $G(H)$ its set of grouplikes. Recall that $\Delta(g)=g \tensor g$ for all $g \in G(H)$. For $g,h \in G(H)$, an element $x \in H$ is \emph{$(g,h)$-skew-primitive} if $\Delta(x)=g \tensor x + x \tensor h$.

\begin{definition}[\cite{ASqls}]
\label{defn.boson}
Let $\theta \in \NN$, $G$ a finite abelian group,
$\gu = g_1,\hdots,g_\theta \in G$, and $\cu = \chi_1,\hdots,\chi_\theta \in \widehat{G}$.
Set $m_i=\ord(\chi_i(g_i)) \geq 2$.
A \emph{bosonization of a quantum linear space over $G$} (QLS for short) is the Hopf algebra $B(G,\gu,\cu)$ generated over $\kk$ by grouplikes $G$ and the $(g_i,1)$-skew-primitive elements $x_1,\hdots,x_\theta$ with relations
\[ gx_i = \chi_i(g)x_ig, \qquad x_i^{m_i} = 0, \qquad x_ix_j = \chi_j(g_i)x_jx_i,\]
for all $g \in G$ and all $i,j \in \{1,\hdots,\theta\}$ with $i \neq j$. 
We refer to $\theta$ as the \emph{rank} of $B(G,\gu,\cu)$.
\end{definition}
\noindent Note that the set $\gu$ appearing above need not be a generating set for $G$. Indeed, we even allow repetitions amongst the $g_i$.

Formally, a \emph{quantum linear space} $\cR(G,\gu,\cu)$ is the braided Hopf algebra in  ${}^{\kk G}_{\kk G} \mathcal{YD}$ generated by the $x_i$ with the relations above. This is in fact a sub-Hopf algebra of the $B(G,\gu,\cu)$ defined above. However, in order to avoid excessive terminology we reserve the acronym QLS for the $B(G,\gu,\cu)$.

Within the class of QLS, we single out the rank 1 case for particular attention.

\begin{definition}
Let $\alpha \in \kk$, $n,m \in \NN$ such that $m \mid n$, and let
$\lambda$ be a primitive $m$th root of unity.
The \emph{generalized Taft algebra} $T_n(\lambda,m,\alpha)$ is generated by a grouplike
element $g$ and a $(g,1)$-skew primitive element $x$ subject to the relations
\[ g^n = 1, \quad x^m = \alpha(g^m-1), \quad gx=\lambda xg.\]
\end{definition}
\noindent 
The (standard) Taft algebra is a generalized Taft algebra in the case $m=n$.
For a QLS $B(G,\gu,\cu)$, we denote by $B_i$ the subalgebra of $B$ generated by $\{g_i,x_i\}$.
It is clear that $B_i \iso T_{|g_i|}(\lambda_i,m_i,0)$ where $\lambda_i=\chi_i(g_i) \in \kk^\times$ has order $m_i$.

This paper is primarily concerned with studying smash products between an algebra and a QLS. This may be thought of as extending Bergen's work to higher rank and in many cases we are able to inductively build results starting in rank 1. On the other hand, we allow $G$ to be an arbitrary finite abelian group (and not just a cyclic group), so in certain cases it will be advantageous to restate some of Bergen's results in terms of group characters.


\begin{definition}
\label{defn.smash}
We say a Hopf algebra $H$ \emph{acts on} an algebra $R$ (from the left)
if $R$ is a left $H$-module via $h\otimes r\mapsto h\cdot r$,
$h \cdot 1_R = \varepsilon(h)1_R$, and $h \cdot (rr') = \sum (h_1 \cdot r)(h_2 \cdot r')$
for all $h \in H$ and $r,r' \in R$.
Alternatively, we say $R$ is a \emph{(left) $H$-module algebra}.
In this case, the \emph{smash product algebra} $R\# H$ is $R \tensor H$ as a $\kk$-vector space, 
with elements denoted by $r\#h$ for $r \in R$ and $h \in H$, and multiplication given by
\[ (r\#h)(r'\#h') = \sum r(h_1 \cdot r')\#h_2h' \quad\text{for all } r,r' \in R \text{ and } h,h' \in H,\]
where the summand is written in Sweedler notation.
\end{definition}

Let $R$ be an algebra and $\sigma \in \Aut(R)$. A $\kk$-linear map $\delta:R \to R$ is a \emph{$\sigma$-derivation} if it obeys the twisted Leibniz rule: $\delta(rs)=\sigma(r)\delta(s)+\delta(r)s$ for all $r,s \in R$.
If $B$ is a QLS and $R$ a $B$-module algebra, then each $g_i$ induces an automorphism $\sigma_i$ of $R$ and each $x_i$ induces a $\sigma_i$-derivation $\delta_i$ of $R$.
To avoid excessive notation, we simply use $g_i$ and $x_i$, respectively, for the automorphisms/skew-derivations.
As such, we typically write $g \cdot r$ and $x_i \cdot r$ as $g(r)$ and $x_i(r)$, respectively.

We define the following invariant subalgebras of $R$,
\begin{align*}
	R^{\grp{x_i}} &= \{ r \in R : x_i(r)=0 \} & R^{\grp{g_i}} &= \{ r \in R : g_i(r)=r \} & & \\
	R^{x} &= \bigcap_{i=1}^\theta R^{\grp{x_i}} & R^G &= \bigcap_{i=1}^\theta R^{\grp{g_i}} & R^B = R^x \cap R^G.
\end{align*}
For $k \in \{1,\hdots,\theta\}$, we set
\[ R_k := \bigcap_{i=1}^k R^{\grp{x_i}} \]
so that $R_\theta = R^x$. Occasionally, we will let $R_0=R^{\grp{x_0}}=R$. 

For an ideal of $I$ of $R$, set $I^* = I \cap R^*$ where $*$ is one of the superscripts above.
A subset $A$ of $R$ is said to be \emph{$B$-stable} if $h(A) \subset A$ for all $h \in B$.
This is equivalent to $x_k(A) \subset A$ for all $k \in \{1,\hdots,\theta\}$ and $g(A) \subset A$ for all $g \in G$.
If $I$ is a $B$-stable ideal of $R$, then $I^B$ is an ideal of $R^B$.

Throughout, unless otherwise stated, 
let $B=B(G,\gu,\cu)$ denote a rank $\theta$ QLS and $R$ a $B$-module algebra.
Set $S=R\#\cR(G,\gu,\cu)$ to be the subalgebra of $R\# B$ generated by
$r\# 1$ for all $r \in R$ and $1 \# x_k$ for all $k \in \{1,\hdots,\theta\}$.
Then $R\#B = S\# \kk G$ where $\kk G$ is the group algebra of $G$.
We frequently drop the $\#$ notation when writing elements, especially in $S$. For $\alpha \in \NN^\theta$, we denote by $x^\alpha$ the monomial $x_1^{\alpha_1} \cdots x_\theta^{\alpha_\theta}$. Thus, we may regard elements in $S$ as polynomials in the $x^\alpha$ with coefficients in $R$.

For $k \in \{1,\hdots,\theta\}$, set $S_k$ to be the subalgebra of $S$ generated by $R$ and $\{x_1,\hdots,x_k\}$, with $S_0=R$. That is, for $k < \theta$, $S_{k+1} = S_k\#\cR(G,\{g_{k+1}\},\{x_{k+1}\})$ where the action of $x_{k+1}$ is extended to $S_k$ by setting $x_{k+1}(x_i)=0$ for $i \leq k$.

In Bergen's work, {\it pseudo-traces} are defined in terms of the root of unity associated to a particular Taft algebra. In order to work in more generality, we will reinterpret this in terms of characters. However, the idempotents appearing in Lemma \ref{lem.ideps} below are distinct from those in \cite{B}.

\begin{lemma}
\label{lem.ideps}
For each $\chi \in \widehat{G}$, define
\[ e_\chi = \frac{1}{|G|} \sum_{g \in G} \chi(g) g \in B.\]
\begin{enumerate}
\item The $e_\chi$ are orthogonal idempotents of $B$ such that $\sum_{\chi \in \widehat{G}} e_\chi = 1$.
\item For all $h \in G$, $h e_\chi = e_\chi h = \chi(h\inv) e_\chi$.
\item For all $x_i$, $e_\chi x_i = x_i e_{\chi \tensor \chi_i}$.
\end{enumerate}
\end{lemma}

\begin{proof}
The first two statements are well-known.
For (3) we need only note the following,
\[ e_\chi x_i 
	= \frac{1}{|G|} \sum_{g \in G} \chi(g) gx_i
	= \frac{1}{|G|} \sum_{g \in G} \chi(g) \chi_i(g)x_i g
	= x_i \frac{1}{|G|} \sum_{g \in G} (\chi \tensor \chi_i)(g) g
	= x_i e_{\chi \tensor \chi_i}. \qedhere\]
\end{proof}

For a rank $\theta$ QLS we set \[ \bx = \prod_{i=1}^\theta x_i^{m_i-1}\] and
for each $\chi \in \widehat{G}$ we define the element $t_\chi \in B$ by 
\[ t_\chi := e_\chi \bx.\]
Throughout, we let $\chi_0$ denote the trivial representation, $e_0 = e_{\chi_0}$, and $t_0 = t_{\chi_0}$.
An easy computation shows that $gt_0 = t_0 = \ep(g)t_0$ for all $g \in G$ and
$x_i t_0 = 0 = \ep(x_i)t_0$ for all $i$ so that $t_0$ is a left integral for $B$.

We establish two main results in this work. The first concerns semiprimeness of smash products.

\begin{thm}[Theorem {\ref{thm.sp1}}]
\label{thm.intro1}
The smash product $R\# B$ is semiprime if and only if, after possibly reordering,
$\bigcap_{i=1}^\theta R^{\grp{x_i}}$ is semiprime,
$x_\theta(L) \neq 0$ for all nonzero left ideals $L$ of $R$,
and for all $k \in \{1,\hdots,\theta-1\}$,
$x_k(L) \neq 0$ for all nonzero left ideals $L$ of
$\bigcap_{j=k+1}^\theta R^{\grp{x_j}}$.
\end{thm}

Theorem \ref{thm.intro1} reduces to \cite[Theorem 4]{B} in the case that $B$ is a Taft algebra.
Similarly, we extend Bergen's result \cite[Theorem 8]{B} to the case of an arbitrary QLS.

\begin{thm}[Theorem {\ref{thm.prime1}}]
\label{thm.intro2}
The smash product $R\#B$ is prime if and only if $t_\chi(L)$ has zero left annihilator
for every $B$-stable left ideal $L \neq 0$ of $R$ and $\chi \in \widehat{G}$.
\end{thm}


In practice, it will be easier to apply the following corollary to domains

\begin{cor}[Corollary {\ref{cor.prime3}}]
The smash product $R\#B$ is prime if and only $t_\chi(R) \neq 0$ for all $\chi \in \widehat{G}$
if and only if, after possibly reordering, $x_{k+1}(R_k) \neq 0$ for $k \in \{0,\hdots,\theta-1\}$
and for each $\chi \in \widehat{G}$ there exists $a \in R^x$ such that $g(a) = \chi(g\inv) a$ for all $g \in G$.
\end{cor}

Beyond the self-contained ring-theoretic interests of these results, there is hope that this paper is a crucial step in providing criteria for actions of \emph{all} finite-dimensional pointed Hopf algebras.
As in \cite{GWY}, it is expected that these results will have applications in the study of invariants of algebras under the action of pointed Hopf algebras.

Our examples feature QLS actions on certain noncommutative algebras that we define now.
Further examples of QLS actions on quantum affine spaces are explored in \cite{CG}.

\begin{definition}
\label{defn.qas}
Let $\bp=(p_{ij}) \in M_n(\kk^\times)$ such that $p_{ii}=1$ and $p_{ij}=p_{ji}\inv$ for all $i \neq j$.
The \emph{quantum affine space} associated to $\bp$ is the $\kk$-algebra
\[ \kk_\bp[u_1,\hdots,u_n] := \kk\langle u_1,\hdots,u_n : u_iu_j=p_{ij}u_ju_i\rangle,\]
The simplest such algebra, when $n=2$, is simply called a 
\emph{quantum plane} and the single nontrivial parameter $p_{ij}$ is denoted by $p$.
\end{definition}

\section{Semiprimeness of smash products}

In this section we establish criteria for a smash product between
an associative algebra and a QLS to be semiprime.
Our first results generalize \cite[Theorem 4]{B} to the case of a QLS.
We begin with a basic result that will be necessary in our computation.

\begin{lemma}
\label{lem.nideal}
Let $R$ be a ring, $A$ a finite nonempty subset of $\Aut(R)$, and $I$ a nilpotent ideal of $R$.
Then $I' = \sum_{\sigma \in A} \sigma(I)$ is a nilpotent ideal of $R$.
\end{lemma}
\begin{proof}
Suppose $I^n = 0$ for some $n \in \NN$.
If $|A|=1$, then $I'=\sigma(I)$ for $\sigma \in A$ and clearly $I'$ is nilpotent, so assume $|A|>1$.

Set $W_k = \{ \sigma_1(I)\sigma_2(I)\cdots \sigma_k(I) : \sigma_i \in A\}$. Then
$(I')^n = \sum_{w \in W_n} w$. 
If $w \in W_n$ satisfies $\sigma_i=\sigma_j$ for all $i,j \leq n$ in its expansion, then $w=0$.
Moreover, $w \subset \sigma_i(I)$ for every $\sigma_i$ appearing in its expansion.
Choose $\tau \in A$. Then for every nonzero $w \in W_n$ in which $\tau$ appears in its expansion, $w \subset \tau'(I)$ for some $\tau' \in A$, $\tau' \neq \tau$. Set $A'= A\backslash\{\tau\}$. Then $(I')^n \subset \sum_{\sigma \in A'} \sigma(I)$. Continuing in this way we arrive at $(I')^m = \sigma(I)$ for some $\sigma \in A$ and some $m\gg 0$.
\end{proof}

\begin{lemma}
\label{lem.sp1}
The smash product $R\# B$ is semiprime if and only if $S$ is semiprime
\end{lemma}

\begin{proof}
Suppose $R\#B$ is semiprime and $I$ is a nilpotent ideal of $S$.
Since $G$ acts on $S$ by automorphism, then $g(I)$ is an ideal of $S$ for every $g \in G$.
Set $I' = \sum_{g \in G} g(I)$, whence $I'$ is a nilpotent $G$-stable ideal of $S$ by Lemma \ref{lem.nideal}.
Thus, \[(S\#\kk G)I' = S\#(I' \kk G) = (SI') \#\kk G = (I'S) \# \kk G = I'(S\#\kk G),\]
so $I'(S\#\kk G)$ is a nilpotent ideal of $S\#\kk G=R\#B$.
Since $R\#B$ is semiprime, then $I'=0$.
It follows that $I=0$ and so $S$ is also semiprime.
Conversely, assume $S$ is semiprime. 
Since $\chr\kk=0$, then $R\#B = S\# \kk G$ is semiprime by \cite[Theorem 7]{FM}.
\end{proof}

Bergen proved the following lemma in the Taft algebra case.
As the relation between the orders of $g$ and $x$ do not factor into the computations,
one may directly port results over to the more general context.

\begin{lemma}
\label{lem.sp2}
Let $T=T_n(\lambda,m,0)$ and let $R$ be a $T$-module algebra.
Then $R\# T$ is semiprime if and only if $R^{\grp{x}}$ is semiprime
and $x(L) \neq 0$ for every nonzero left ideal $L$ of $R$.
\end{lemma}

\begin{proof}
By Lemma \ref{lem.sp1}, 
it suffices to prove that $R\# T$ is semiprime if and only if $S=R\# \cR(G,\{g\},\{x\})$ is semiprime.
Applying the proof of \cite[Theorem 4]{B}, $S$ is semiprime if and only if $R^{\grp{x}}$ is semiprime
and $x(L) \neq 0$ for every nonzero left ideal $L$ of $R$.
\end{proof}

We are now ready for our main result of the section.
The idea is to inductively build $S$ and use repeated applications of Lemma \ref{lem.sp2}.

\begin{theorem}
\label{thm.sp1}
The smash product $R\# B$ is semiprime if and only if, after possibly reordering,
\begin{enumerate}
\item $\bigcap_{i=1}^\theta R^{\grp{x_i}}$ is semiprime,
\item $x_\theta(L) \neq 0$ for all nonzero left ideals $L$ of $R$, and
\item for all $k \in \{1,\hdots,\theta-1\}$,
$x_k(L) \neq 0$ for all nonzero left ideals $L$ of
$\bigcap_{j=k+1}^\theta R^{\grp{x_j}}$.
\end{enumerate}
\end{theorem}

\begin{proof}
By Lemma \ref{lem.sp1}, $R\#B$ is semiprime if and only if $S$ is semiprime.
Since $S \iso S_{\theta-1} \# \cR(G,\{g_\theta\},\{x_\theta\})$, then by Lemma \ref{lem.sp2},
$S$ is semiprime if and only if $S_{\theta-1}^{\grp{x_\theta}}$ is semiprime and $x_\theta(L)\neq 0$
for all nonzero left ideals $L$ of $S_{\theta-1}$.

Let $L$ be a nonzero left ideal of $S_{\theta-1}$.
We denote by $L'$ the set of leading coefficients of elements of $L$ under the standard lexicographic ordering ($x_1 > x_2 > \cdots > x_{\theta-1}$). 
By standard arguments related to the (skew) Hilbert Basis Theorem (see, e.g., \cite{GW}), $L'$ is a left ideal of $R$ and, by hypothesis, $L'\neq 0$.
If $x_\theta(L)=0$, then $x_k(L')=0$.
On the other hand, if $L$ is a nonzero ideal of $R$, then $L$ extends to an ideal $L''$ of $S$
whose coefficients are elements of $L$.
Obviously, if $x_\theta(L)=0$, then $x_\theta(L'')=0$ and so 
we conclude that $S$ is semiprime if and only if
$S_{\theta-1}^{\grp{x_\theta}}$ is semiprime and $x_\theta(L)\neq 0$
for all nonzero left ideals $L$ of $R$.

Let $p \in S_{\theta-1}$ and write $p = \sum_i r_i x^{\alpha_i}$, with $r_i \in R$ and $\alpha_i \in \NN^{\theta-1}$.
Then $x_\theta(p) = \sum x_\theta(r_i) x^{\alpha_i}$ and so 
$x_\theta(p)=0$ if and only if $r_i \in R^{\grp{x_\theta}}$ for all $i$. That is, 
\begin{align*}
S_{\theta-1}^{\grp{x_\theta}} 
	&= R^{\grp{x_\theta}}\#\cR(G,\{g_1,\hdots,g_{\theta-1}\},\{x_1,\hdots,x_{\theta-1}\}) \\
	&= \left( R^{\grp{x_\theta}}\#\cR(G,\{g_1,\hdots,g_{\theta-2}\},\{x_1,\hdots,x_{\theta-2}\}) \right) \# \cR(G,\{g_{\theta-1}\},\{x_{\theta-1}\}).
\end{align*}
Set $R'=R^{\grp{x_\theta}}\#\cR(G,\{g_1,\hdots,g_{\theta-2}\},\{x_1,\hdots,x_{\theta-2}\})$.
Applying Lemma \ref{lem.sp2} again, we have that $S_{\theta-1}^{\grp{x_\theta}}$ is semiprime if and only if
$\left( R' \right)^{\grp{x_{\theta-1}}}$ 
is semiprime and $x_{\theta-1}(L) \neq 0$ for all nonzero left ideals $L$ of $R'$.
Using the trick above we can restate the last condition as $x_{\theta-1}(L) \neq 0$ for all nonzero left ideals $L$ of $R^{\grp{x_\theta}}$.

Furthermore, using our above calculations we see that
\begin{align*}
(R')^{\grp{x_{\theta-1}}} 
	&=  \left( R^{\grp{x_\theta}} \right)^{\grp{x_{\theta-1}}}\#\cR(G,\{g_1,\hdots,g_{\theta-2}\},\{x_1,\hdots,x_{\theta-2}\}) \\
	&= \left( R^{\grp{x_\theta}} \cap R^{\grp{x_{\theta-1}}} \right)\#\cR(G,\{g_1,\hdots,g_{\theta-2}\},\{x_1,\hdots,x_{\theta-2}\}).
\end{align*}
Continuing in this way, we establish the theorem.
\end{proof}



We give an example below in which condition (3) fails and the smash product is not semiprime.

\begin{example}
\label{ex.nsp}
Let $B=B(G,\{g_1,g_2\},\{\chi_1,\chi_2\})$ where $G = \grp{g_1} \times \grp{g_2} \iso \ZZ_n \times \ZZ_n$ 
and $A=\kk_p[u_1,u_2]$ with $p$ a primitive $k$th root of unity where $k \mid n$ and $k>1$ (see Definition \ref{defn.qas}).
Let $\omega$ be a primitive $n$th root of unity and recall that $\lambda_i = \chi_i(g_i)$ for $i=1,2$.
Set $\chi_1(g_1)=\chi_2(g_1)= \omega$ and $\chi_2(g_2)=\chi_1(g_2) = \omega\inv$.
By \cite[Proposition 2.1]{GWY}, 
\[ g_i(u_1)=p u_1, \quad g_i(u_2)=\lambda_i\inv p u_2, \quad x_i(u_1)=0, \quad x_i(u_2)=u_1,\]
for $i=1,2$. A straightforward check shows that this defines an action of $B$ on $A$.
Furthermore, by \cite[Lemma 2.1]{GWY}, $R^{\grp{x_1}}=R^{\grp{x_2}}=\kk[u_1,u_2^n]$, 
and so $x_1(L)=0$ for any left ideal $L$ of $\kk[u_1,u_2^n]$.
Thus, property (3) in Theorem \ref{thm.sp1} fails in this case and so $A\# B$ is not semiprime.
\end{example}

The criteria of Theorem \ref{thm.sp1} simplify considerably in the case that $R$ is a domain.
We then apply this to a case where the smash product is semiprime.

\begin{corollary}
\label{cor.sp1}
\begin{enumerate}
\item If $R$ is semiprime, then $R\# B$ is semiprime if and only if, after possibly reordering,
\begin{enumerate}
\item $\bigcap_{i=1}^\theta R^{\grp{x_i}}$ is semiprime,
\item $x_\theta(I) \neq 0$ for all nonzero ideals $I$ of $R$, and 
\item for all $k \in \{1,\hdots,\theta-1\}$,
$x_k(I) \neq 0$ for all nonzero ideals $I$ of $\bigcap_{j=k+1}^\theta R^{\grp{x_j}}$.
\end{enumerate}
\item If $R$ is prime, then $R\# B$ is semiprime if and only if, after possibly reordering,
$\bigcap_{i=1}^\theta R^{\grp{x_i}}$ is semiprime,
$x_\theta(R) \neq 0$, 
and for all $k \in \{1,\hdots,\theta-1\}$, $x_k\left( \bigcap_{j=k+1}^\theta R^{\grp{x_j}} \right) \neq 0$.
\item If $R$ is a domain, then $R\# B$ is semiprime if and only if, after possibly reordering,
$x_\theta(R) \neq 0$ and for all $k \in \{1,\hdots,\theta-1\}$,
$x_k\left( \bigcap_{j=k+1}^\theta R^{\grp{x_j}} \right) \neq 0$.
\end{enumerate}
\end{corollary}

\begin{proof}
We will prove (1) following \cite[Corollary 5]{B}. Parts (2) and (3) follow similarly.

If we assume $R\# B$ is semiprime, then conditions (a)-(c) must hold by Theorem \ref{thm.sp1}.
Conversely, assume conditions (a)-(c) hold and let $L$ be a nonzero left ideal of $R$ such that $x_\theta(L)=0$.
Then $RL \subset L$ so 
\begin{align}
\label{eq.ann}
0 = x_\theta(RL) = g_\theta(R)x_\theta(L) + x_\theta(R)L = x_\theta(R)L,
\end{align}
whence $( Lx_\theta(R) )^2 = L(x_\theta(R)L)x_\theta(R) = 0$.
As $Lx_\theta(R)$ is a nilpotent left ideal of the semiprime ring $R$, then $Lx_\theta(R)=0$ and so
\[ x_\theta(LR) = g_\theta(L)x_\theta(R) + x_\theta(L)R = g_\theta(L)x_\theta(g_\theta(R)) = g_\theta(Lx_\theta(R)) = 0.\]
But $0 \neq L^2 \subset LR$ and $LR$ is a two-sided ideal of the semiprime ring $R$,
contradicting our hypothesis.
Thus, $x_\theta(L)\neq 0$. A similar proof works in the subrings as well.
\end{proof}



\begin{example}
\label{ex.sp}
Let $A=\kk_{\bp}[u_1,u_2,u_3]$ be a quantum affine space with
\[
\bp = \left(\begin{smallmatrix}
1 	& p 	& p\inv \\
p\inv & 1	& p\inv \\
p	& p	& 1
\end{smallmatrix}\right)
\]
for some $p$ a primitive $n$th root of unity, $n>1$.
Let $G=\grp{g_1} \times \grp{g_2} \iso \ZZ_n \times \ZZ_n$ and consider
the rank 2 QLS $B=B(G,\gu,\cu)$ with data
\[ \chi_1(g_1) = p\inv, \quad \chi_1(g_2)=p, \quad \chi_2(g_1)=p\inv, \quad \chi_2(g_2)=p\inv.\]
We define an action of $B$ on $A$ by setting
$g_1=\diag(p,p^2, 1)$, $g_2=\diag(p, 1, p^2)$,
$x_1(u_2)=u_1$, $x_2(u_3)=u_1$, and $x_i(u_j)=0$ for all other $i,j$.
As $A$ is a domain, we may apply Corollary \ref{cor.sp1} (3).
Note that $x_2(u_3)=u_1 \neq 0$, so $x_2(A) \neq 0$.
Moreover, $A^{\grp{x_2}} = \kk[u_1,u_2,u_3^n]$ and $x_1(u_2) = u_1 \neq 0$ so $x_1(A^{\grp{x_2}}) \neq 0$.
Thus, $A\# B$ is semiprime.
%
\end{example}

\section{Primeness of smash products}

In this section, we give criteria for a smash product $R\#B$ to be prime.
The next result is another instance where we may apply
Bergen's result almost directly, this time \cite[Lemma 3]{B}.

\begin{lemma}
\label{lem.nilp}
Let $T=T_n(\lambda,m,0)$ and let $R$ be a $T$-module algebra.
If $R^{\grp{x}}$ is semiprime and $L$ is a left ideal of $R$ such that $x(R)\neq 0$,
then $x^{m-1}(L)$ is a left ideal of $R^{\grp{x}}$ that is not nilpotent.
\end{lemma}

Our next goal is to generalize the previous lemma to the QLS setting.
The following two results establish non-nilpotency of certain ideals in $R$.

\begin{lemma}
\label{lem.ideals0}
Suppose, after possibly reordering, that each $R_k$ is semiprime, $k \in \{1,\hdots,\theta\}$,
and $L$ is a left ideal of $R$ such that $x_{k+1}(L_k) \neq 0$ where $L_0 = L$ and
$L_k = x_k^{n_k-1}(L_{k-1})$ for $k \geq 1$.
Then $\bx(L)$ is a left ideal of $R^x$ that is not nilpotent.
\end{lemma}

\begin{proof}
Applying Lemma \ref{lem.nilp} to the left ideal $L$ and the semprime ring $R_1$,
we have that $x_1(L)\neq 0$ implies $x_1^{n_1-1}(L)$ is a left ideal of $R^{\grp{x_1}}$ that is not nilpotent.
Suppose for some $k\geq 1$ that $L_k$ is a left ideal of $R_k$ such that $x_{k+1}(L_k) \neq 0$.
As $R_{k+1}$ is semiprime, then applying Lemma \ref{lem.nilp} again we have that 
$x_{k+1}^{n_{k+1}-1}(L_k)$ is a left ideal of $R_{k+1}$ that is not nilpotent.
The result follows by induction.
\end{proof}

Recall that an ideal $I$ of $R$ is $B$-stable if $x_k(I) \subset I$ for all $k\in\{1,\hdots,\theta\}$ and $g(I) \subset I$ for all $g \in G$.

\begin{lemma}
\label{lem.ideals1}
\begin{enumerate}
\item If, after possibly reordering, each $R_k$ is semiprime, $k\in \{0,\hdots,\theta-1\}$, 
and $I \neq 0$ is a $B$-stable ideal of $R$, then $I^B$ is not nilpotent.
\item If $R$ is reduced and $A \neq 0$ is a $B$-stable subring of $R$, then $A^B$ is not nilpotent.
\end{enumerate}
\end{lemma}

\begin{proof}
(1) As $I$ is $B$-stable, $x_1(I) \subset I$.
Since we assume that $R=R_0$ is semiprime, then it follows from \cite[Theorem 5]{BergenG} that  
$I_1 = I^{\grp{x_1}}$ is a non-nilpotent ideal of $R_1$.

Suppose inductively that $I_k = I \cap R_k$ is a non-nilpotent ideal of $R_k$.
Recall that $x_i \in B$ is a $g_i$-skew derivation for each $i \in \{1,\hdots,k\}$.
If $a \in I_k$, then for all $i \in \{1,\hdots,k\}$, \[x_i(x_{k+1}(a)) = \chi_{k+1}(g_i)x_{k+1}(x_i(a)) = 0.\]
Thus, $x_{k+1}(I_k) \subset I_k$.
As $R_k$ is semiprime then again by \cite[Theorem 5]{BergenG} we have that
$I_{k+1}=I_k^{\grp{x_{k+1}}}$ is a non-nilpotent ideal of $R_{k+1}$.
Thus, by induction $I^x = I_\theta$ is a non-nilpotent ideal of $R^x = R_\theta$.


The elements of $G$ skew-commute with the $x_i$,
so $I^x$ is a $G$-stable ideal of $R^x$.
By \cite{BI}, this implies that $I^B = (I^x)^G$ is not nilpotent.

(2) As $A$ is $B$-stable and each $x_i$ is nilpotent, then $A^x$ is nonzero. 
In this case, $R$ has no nonzero nilpotent elements and so the same must hold true for the subring $A^x$. Again applying results from \cite{BI}, we have $A^B = (A^x)^G$ is not nilpotent.
\end{proof}

A $B$-module algebra $R$ is said to be \emph{$B$-prime} 
if $IJ=0$ implies $I=0$ or $J=0$ for every pair of $B$-stable ideals $I,J$ of $R$.

\begin{lemma}
\label{lem.ideals2}
Suppose $R$ is reduced.
If $R^B$ is prime, then $R$ is $B$-prime.
\end{lemma}

\begin{proof}
Suppose $I,J$ are nonzero $B$-stable ideals of $R$.
Then $I^B,J^B$ are nonzero $B$-stable ideals of $R^B$ by Lemma \ref{lem.ideals1} (2).
Then $0 \neq I^BJ^B \subset IJ$ by the primeness of $R^B$.
Hence, $IJ \neq 0$ and so $R$ is $B$-prime.
\end{proof}

The next set of results generalize \cite[Lemma 6]{B}.

\begin{lemma}
\label{lem.stble}
For every two-sided ideal $I \neq 0$ of $R\# B$,
there exists a $B$-stable left ideal $L\neq 0$ of $R$ such that $Lt_\chi \subset I$,
for some $\chi \in \widehat{G}$.
\end{lemma}

\begin{proof}
Let $I$ be a nonzero ideal of $R\# B$ and let $m = x_1^{\alpha_1}\cdots x_\theta^{\alpha_\theta}$
be a monomial of maximal degree such that $Im \neq 0$.
As $x_i\bx = 0$ for all $i$ but $\bx \neq 0$, then $0 \neq Im \subset (R\#\kk G)\bx$.
Since $\sum_{\chi \in \widehat{G}} e_\chi = 1$, then there exists $\chi \in \widehat{G}$ such that 
\[ 0 \neq Ime_\chi \subset (R\#\kk G)\bx e_\chi.\]
By Lemma \ref{lem.ideps} (3), $\bx e_{\chi} = e_{\chi'}\bx$ for some $\chi' \in \widehat{G}$.
Thus, $(R\#\kk G)\bx e_\chi = (R\#\kk G)e_{\chi'}\bx = (R\#\kk G)t_{\chi'}$.
Orthogonality of the idempotents implies $(\kk G)e_{\chi'} = \kk e_{\chi'}$ and so
the elements of $(R\#\kk G)t_{\chi'}$ (and hence of $Ime_\chi$) are of the form $at_{\chi'}$ for $a \in R$.
It follows that $I$ contains nonzero elements of this form.

Hence, there exists $\chi \in \widehat{G}$ such that
$L = \{ a \in R : at_{\chi} \in I \}$
is a nonzero left ideal of $R$. We claim that $L$ is stable under the action of the $G$ and the $x_i$.
Let $a \in L$ and choose a corresponding $at_{\chi} \in I$. For $h \in G$, $\chi(h)hat_{\chi} \in I$
and so by Lemma \ref{lem.ideps}, 
\[ \chi(h)ha t_{\chi}
	= \chi(h)h(a)ht_{\chi}
	= h(a) \chi(h)\chi(h\inv) t_{\chi}
	= h(a) t_{\chi}.\]
Thus, $h(a) \in L$. Similarly, for each $x_i$,
\[ x_ia t_{\chi} = (x_ia)t_{\chi} = ( g_i(a)x_i + x_i(a))t_{\chi}
	= (g_i(a)e_{\chi \tensor \chi_i\inv}x_i\bx) + x_i(a)t_{\chi}
	= x_i(a)t_{\chi}.\]
Thus, $x_i(a) \in L$, completing the proof.
\end{proof}

As observed in the introduction, $t_0$ is a left integral of $B$.
Thus, for any $h \in B$ and $r \in R$, $hrt_0 = h(r)t_0$.
In particular, $t_\chi r t_0 = t_\chi(r)t_0$.

For each $\chi \in \widehat{G}$, define the map
$\pi_\chi: R\#B \rightarrow R\#B$ that is the identity on $R$ and the $x_i$, 
and $\pi_\chi(h)=\chi(h)h$ for all $h \in G$.
It is clear that $\pi_\chi$ restricts to an automorphism of $\kk G$ and to an automorphism of $B$ as 
it preserves the defining relations from Definition \ref{defn.boson}.
Moreover, for all $r \in R$, $\pi_\chi$ preserves the relations $gr=g(r)g$ for all $g \in G$
as well as $x_kr=g_k(r)x_k + x_k(r)$ for all $k \in \{1,\hdots,\theta\}$ and all $r \in R$.
Thus, $\pi_\chi$ extends linearly to an automorphism of $R\# B$.
If $\chi' \in \widehat{G}$, then
\[ \pi_\chi(e_{\chi'}) = \frac{1}{|G|} \sum_{g \in G} \chi'(g) \chi(g) g = e_{\chi' \tensor \chi}.\]
Consequently, if $\chi_0$ is the trivial character, then $\pi_{\chi\inv}(e_\chi) = e_0$ and $\pi_{\chi\inv}(t_{\chi})=t_{\chi_0}$.

\begin{theorem}
\label{thm.prime1}
The smash product $R\#B$ is prime if and only if $t_\chi(L)$ has zero left annihilator
for every $B$-stable left ideal $L \neq 0$ of $R$ and $\chi \in \widehat{G}$.
\end{theorem}

\begin{proof}
Assume $R\# B$ is prime. 
Let $L\neq 0$ be a $B$-stable left ideal of $R$ and let $\chi \in \widehat{G}$.
Suppose $a \in R$ such that $at_\chi(L)=0$.
By Lemma \ref{lem.ideps} (2,3), $Le_{\chi_0}\bx$ is a nonzero left ideal of $R\# B$.
Then for any $r \in R$, $t_\chi r t_{\chi_0} = t_\chi(r) t_{\chi_0}$. Thus,
$(a t_\chi)(L t_{\chi_0}) = at_\chi(L) t_{\chi_0} = 0$,
whence $at_\chi$ is the left annihilator of $Le_{\chi_0}\bx$.
By primeness, $at_\chi = 0$ so $a=0$.

Conversely, assume $R\#B$ is not prime and let $I,J$ be nonzero
ideals in $R\# B$ such that $IJ=0$.
By Lemma \ref{lem.stble}, there exists nonzero $B$-stable left ideals $L,L'$ of $R$
and $\chi,\chi' \in \widehat{G}$ such that $Lt_\chi \subset J$ and $L't_{\chi'} \subset I$.
Applying $\pi_{\chi\inv}$ to the first inclusion gives $Lt_{\chi_0} \subset \pi_{\chi\inv}(J)$.
Since $\pi_{\chi\inv}$ is an automorphism, then $\pi_{\chi\inv}(I)\pi_{\chi\inv}(J)=0$ so there 
is no loss in assuming $Lt_{\chi_0} \subset J$. Now
\[ L' t_{\chi'}(L)t_{\chi_0} = (L't_{\chi'})(Lt_{\chi_0}) \subset IJ = 0.\]
This completes the proof as $L' \neq 0$ and $L'$ lies in the annihilator of $t_{\chi'}(L)$.
\end{proof}

\begin{corollary}
\label{cor.prime2}
Let $R$ be a reduced $B$-module algebra. The following are equivalent.
\begin{enumerate}
\item $R\# B$ is prime.
\item $R^B$ is prime and $t_\chi(R) \neq 0$ for all $\chi \in \widehat{G}$.
\item $R$ is $B$-prime and $t_\chi(R) \neq 0$ for all $\chi \in \widehat{G}$.
\item $t_\chi(I)$ has zero left annihilator in $R$ for every nonzero $B$-stable ideal $I$ of $R$
and every $\chi \in \widehat{G}$.
\end{enumerate}
\end{corollary}

\begin{proof}
$(1) \Rightarrow (2)$ 
Let $a,b \in R^B$ be nonzero. Then $Rb$ is a nonzero $B$-stable left ideal of $R$.
By Theorem \ref{thm.prime1}, $at_0(Rb) \neq 0$ and
$0 \neq at_0(Rb) = at_0(R)b \subset aR^B b$.
Thus, $aR^Bb \neq 0$ and so $R^B$ is prime.
Moreover, by Theorem \ref{thm.prime1}, $t_\chi(R) \neq 0$ for all $\chi \in \widehat{G}$.

$(2) \Rightarrow (3)$ This follows from Lemma \ref{lem.ideals2}.

$(3) \Rightarrow (4)$
Let $I$ be a nonzero $B$-stable ideal of $I$.
It suffices to prove that $t_\chi(I) \neq 0$ for all $\chi \in \widehat{G}$.
By \cite[Theorem 2]{yanai}, if $t_\chi$ vanishes on $I$ then it vanishes on all of $R$.
The hypothesis implies this cannot happen.

$(4) \Rightarrow (1)$
Let $L$ be a nonzero $B$-stable left ideal of $R$ and $\chi \in \widehat{G}$.
By Theorem \ref{thm.prime1}, it suffices to prove that $t_\chi(L)$
has zero left annihilator in $R$.
The left/right annihilators are equal by hypothesis and so are two-sided ideals.
Let $U=\ann(L)$ and $V=\ann(U)$.
Since $U$ and $V$ are $B$-stable and $0 \neq L \subset V$, then
\[ Ut_\chi(V) \subset UV = 0,\]
so $U \subset \ann(t_\chi(V))$ and thus $U=0$.
Again by \cite[Theorem 2]{yanai}, if $at_\chi(L)=0$, then $at_\chi(R)=0$.
Thus, our hypothesis implies $a=0$. That is, $\ann(t_\chi(L))=0$, so $R\#B$ is prime.
\end{proof}




We now simplify the criteria above in the case that the base ring is a domain.

\begin{corollary}
\label{cor.prime3}
Let $R$ be a $B$-module algebra that is a domain. The following are equivalent.
\begin{enumerate}
\item $R\#B$ is prime.
\item $t_\chi(R) \neq 0$ for all $\chi \in \widehat{G}$.
\item After possibly reordering $x_{k+1}(R_k) \neq 0$ for $k \in \{0,\hdots,\theta-1\}$
and for each $\chi \in \widehat{G}$ there exists a nonzero $a \in R^x$ such that $g(a) = \chi(g\inv) a$ for all $g \in G$.
\end{enumerate}
\end{corollary}

\begin{proof}
$(1) \Leftrightarrow (2)$ This follows by Corollary \ref{cor.prime2}.

$(2) \Rightarrow (3)$ Let $\chi \in \widehat{G}$, so $t_\chi(R) \neq 0$.
Since $t_\chi = e_\chi \bx$, then  $\bx(R) \neq 0$ and so $x_{k+1}(R_k) \neq 0$ for $k \in \{0,\hdots,\theta-1\}$.
Applying Lemma \ref{lem.ideals0} with $L=R$ gives that $\bx(R)$ is a left ideal of $R^x$ that is not nilpotent, whence $R^x\neq 0$. As the elements of $G$ skew-commute with the $x_i$, then $G$ acts on $R^x$.
For $\psi \in \widehat{G}$ we define 
$(R^x)_\psi = \{ r \in R : g(r)=\psi(r)r \text{ for all } g\in G\}$.
By \cite{pass2},
\[ R^x = \bigoplus_{\psi \in \widehat{G}} (R^x)_\psi.\]
Now if $a \in (R^x)_\psi$ for some $\psi \in \widehat{G}$, then 
\begin{align}
\label{eq.char}
e_\chi(a)
	= \frac{1}{|G|} \sum_{g \in G} \chi(g) g(a)
	= \frac{1}{|G|} \sum_{g \in G} \chi(g) \psi(g) a
	= \frac{1}{|G|} \sum_{g \in G} (\chi \tensor \psi)(g) a.
\end{align}
It follows from standard abelian group character theory that this sum is nonzero if and only if $\psi = \chi\inv$.
Since we assume that $t_\chi(R) \neq 0$, then it follows 
that there exists $a \in (R^x)_{\chi\inv}$, $a \neq 0$, such that $g(a)=\chi(g\inv) a$ for all $g \in G$.

$(3) \Rightarrow (2)$ 
Let $\chi \in \widehat{G}$ and assume that there exists a nonzero
$a \in R^x$ such that $g(a) = \chi(g\inv) a$ for all $g \in G$.
After reordering, $x_{k+1}(R_k) \neq 0$ for $k \in \{0,\hdots,\theta-1\}$.
Then again $\bx(R)$ is a nonzero left ideal of $R^x$ by Lemma \ref{lem.ideals0}.
Since $R^x$ is a domain, then by \cite[Theorem 6]{yanai}, $t_\chi$ vanishes on $R^x$
if and only if it vanishes on $\bx(R)$. As
$t_\chi(R) = e_\chi(\bx(R))$,
then to show $t_\chi \neq 0$ we need only show that $e_\chi$ does not vanish on $R^x$.
Now \eqref{eq.char} shows that $e_\chi(a) \neq 0$.
Consequently $t_\chi \neq 0$.
\end{proof}


\begin{example}
\label{ex.spnp}
Let $A$ and $B$ be as in Example \ref{ex.sp} and recall that $A\# B$ is semiprime.
By computations in that example,
$A^{\grp{x_2}} = \kk[u_1,u_2,u_3^n]$ and $x_1(u_2) = u_1 \neq 0$, so $x_1(A^{\grp{x_2}}) \neq 0$,
thus we have satisfied the first part of Corollary \ref{cor.prime3} (3).
However, we note that $A^x=\kk[u_1,u_2^n,u_3^n]$ and so $g_1(a)=g_2(a)$ for all $a \in A^x$.
It follows that the second part of Corollary \ref{cor.prime3} (3) cannot be satisfied.
Simply choose a character $\chi \in \widehat{G}$ such that $\chi(g_1\inv) \neq \chi(g_2\inv)$.
Hence, $A\# B$ is not prime.
\end{example}

\begin{example}
Let $\tornado$ be a primitive third root of unity and let
$A=\kk_\bp[u_1,u_2,u_3,u_4]$ with
\[
\bp = \left(\begin{smallmatrix}
1 			& \tornado		& \tornado		& 1 \\
\tornado^2	& 1			& \tornado^2	& \tornado			\\
\tornado^2	& \tornado		& 1			& \tornado		\\
1			& \tornado^2	& \tornado^2	& 1
\end{smallmatrix}\right).
\]
Let $G=\grp{g_1} \times \grp{g_2} \iso \ZZ_n \times \ZZ_n$ and consider
the rank 2 QLS $B=B(G,\gu,\cu)$ with data
\[ 
\chi_1(g_1) = \tornado, \quad \chi_1(g_2)=\tornado, \quad 
\chi_2(g_1)=\tornado^2, \quad \chi_2(g_2)=\tornado^2.
\]
Then $B$ acts on $G$ by setting 
$g_1 = \diag(\tornado, 1, 1, \tornado^2)$, 
$g_2 = \diag(1, \tornado^2, \tornado^2, \tornado)$,
$x_1(u_3)=u_1$, $x_2(u_4)=u_1$, and $x_i(u_j)=0$ for all other $i,j$.

Now, as above, we have $A^{\grp{x_1}}=\kk[u_1,u_2,u_3^3,u_4]$.
Since $x_2(u_4)=u_1$, then we have $x_2(A^{\grp{x_1}}) \neq 0$.
Moreover, $A^x=\kk[u_1,u_2,u_3^3,u_4^3]$. Now, 
$g_1(u_1^k u_2^\ell) = \tornado^k u_1^k u_2^\ell$ while
$g_2(u_1^k u_2^\ell) = \tornado^{2\ell} u_1^k u_2^\ell$.
Thus, given $\chi \in \widehat{G}$, we can set $a=u_1^k u_2^\ell$ and we need only choose 
an appropriate value of $k$ and $\ell$ so that $\chi(g_1\inv)=g_1(a)$ and $\chi(g_2\inv)=g_2(a)$.
It follows now from Corollary \ref{cor.prime3} that $A\# B$ is prime.
\end{example}

\subsection*{Acknowledgment}
The author was partially supported by a grant from the 
Miami University Senate Committee on Faculty Research.

\bibliographystyle{plain}

\end{document}